\documentclass[leqno,CJK]{siamltex704}
\usepackage{amssymb,amsmath,graphicx,amscd,mathrsfs}
\usepackage{color,xcolor,comment,dsfont,pifont}
\usepackage{longtable,graphicx,float,amsfonts,amssymb}
\usepackage{hyperref,multirow,placeins}
\numberwithin{equation}{section}
\def\3bar{{|\hspace{-.02in}|\hspace{-.02in}|}}
\def\E{{\mathcal{E}}}
\def\T{{\mathcal{T}}}
\def\Q{{\mathcal{Q}}}
\def\dQ{{\mathbb{Q}}}

\def\b0{\boldsymbol{0}}
\def\bPhi{\boldsymbol{\Phi}}            
\def\sumT{\sum_{T\in\mathcal{T}_h}}     

\def\bw{{\mathbf{w}}}
\def\bu{{\mathbf{u}}}
\def\bv{{\mathbf{v}}}
\def\bl{{\mathbf{L}}}

\def\bh{{\mathbf{H}}}
\def\bi{{\mathbf{I}}}
\def\bn{{\mathbf{n}}}
\def\be{{\mathbf{e}}}
\def\bf{{\mathbf{f}}}

\def\bP{{\mathbf{P}}}

\newtheorem{example}{Example}[section]
\newtheorem{remark}{Remark}[section]
\newtheorem{algorithm1}{Weak Galerkin Algorithm}
\newtheorem{algorithm}{Finite element Algorithm}

\newcommand{\eps}{\varepsilon}
\newcommand{\di}{\text{div}}

\newcommand{\Real}{\mathbb{R}}

\newcommand{\la}{\langle}
\newcommand{\ra}{\rangle_{\partial T}}

\allowdisplaybreaks 

\begin{document}
	
	\title{The Weak Galerkin and Crouzeix-Raviart element method for elastic eigenvalue problems}
	\author{
		Wei Lu, Hehu Xie, and Qilong Zhai
		\thanks{Department of Mathematics, Jilin University, Changchun,
			China. }
		}
	\maketitle
	
	\begin{abstract}
		In this paper, we first introduce an abstract framework to solve the eigenvalue problem by weak Galerkin (WG) method. By the application of the framework, WG method is proved to be locking-free and gives asymptotic lower bounds for the elastic eigenvalue problem. Also, we analyze the lower bound property for Crouzeix-Raviart (CR) element as an extensional work. In the end, we present some numerical experiments to support the theoretical results.
	\end{abstract}
	
	

	\section{Introduction}
Eigenvalue problem, especially the linear elastic eigenvalue problem, has received lots of attention since its broad applications in science and engineering, such as the deformation analysis of compressible and imcompressible elastic meterials \cite{Vogelius1983}. There has been quiet a few numerical methods to solve the elastic eigenvalue problem, such as finite element method \cite{Meddahi2013,Inzunza2023,Zhang2024} and finite diffrence method \cite{Yang2023,Lyu2024}.

The finite element method (FEM) is considered as an efficient way to deal with PDEs due to its straightforward and adaptivity on triangular meshes. It has been applied for many eigenvalue problems, such as Laplacian eigenvalue problem \cite{Liu2013,Grebenkov2013,Luo2012} and Stokes eigenvalue problem \cite{Feng2014,Chen2009}. Nevertheless, when applied to solve the linear elastic eigenvalue problem, many finite element schemes, such as the standard comforming FEM, suffer from the “locking” phenomenon.

“Locking” refers to a phenomenon that, for the linear elasticity problem, the numerical approximations of finite element scheme are depend on the Poisson ratio $\nu$ (i.e. the Lam\'{e} constant $\lambda$). When $\nu$ approaches $\frac{1}{2}$ (i.e. $\lambda$ turns to $\infty$), the elastic material becomes nearly imcompressible. In this case, the numerical solutions of various finite element schemes do not converge \cite{Ainsworth2022,Babuska1992,Ciarlet2022,Bren1992}. But it is not difficult to aviod locking. One possible approach is to use mixed finite element method \cite{Lepe2022,Lepe2019, Meddahi2013}. The main drawback of this approach is the difficulty in ensuring the stability of the mixed element method. Another useful approach is the construction of nonconforming element or virtual element. In \cite{Bi2023}, Zhang et al. used Crouzeix-Raviart element to solve the linear elastic eigenvalue problem with pure displacement boundary. In \cite{Mora2020}, a vitual element was applied to deal with the linear elastic eigenvalue problem under mixed boundary condition. We refer interested readers to \cite{Hansbo2003,Zhang2024,Amigo2023} for more information. Nevertheless, this approach seems to be hard to construct high order or high dimension finite element space.

Not long ago, a new class of nonconforming finite element, called weak Galerkin method, is able to cope with aforementioned problems. The WG method, which was first introduced in \cite{Wang2013a}, features in the application of weak differential operators instead of classic differential operators. Additionally, WG method adopts discontinuous piecewise polynomials on polygonal finite element partitions, making it possible to be extended to high dimension cases and adopt polytopal meshes. So far, WG method has been employed to solve various kinds of PDEs \cite{Mu2013c,Zhu2014,Mu2015,Wang2014,Mu2015b,Lin2014,Chen2016c,Zhai2016}. In particular, WG method is a useful way to solve eigenvalue problems, since it can provide asymptotic lower bounds for the eigenvalues by the use of high order polynomial elements.

It is worth mentioning that, when the eigenvalues are real numbers, it is necessary to obtain both upper bounds and lower bounds of the eigenvalues to get accurate intervals which they belongs to\cite{Luo2012}. Due to the min-max principle \cite{MR1115240}, all conforming FEM can only provide upper bounds for eigenvalues. Compared to the upper bounds, however, the lower bounds of eigenvalues are harder to get. Through post-processing procedures or the construction of nonconforming FEM are the main ways to finish it, but both of them have some shortages \cite{Zhai2019}. As mentioned in the last of preceding paragragh, WG method can also provide lower bounds for eigenvalues, which is achieved by using a function $\gamma(h)$ in the stablization term. Furthermore, inspired by this technique, we prove the nonconforming CR element method is also capable of providing lower bounds for linear elastic eigenvalue problem with mixed boundary under some conditions, which is based on the finished work by Zhang et al. in \cite{Zhang2023}.

The rest of this paper is organized as follows. In Section 2 we state some notations and weak form of the linear elastic eigenvalue problem. Section 3 is devoted to an abstract framework for soving eigenvalue problems by weak Galerkin method. In Section 4, we shall use the framework to derive the error estimates and lower bound property of WG method. In Section 5, we show that the CR element method can also give lower bounds for eigenvalues under some conditions. Some numerical experiments of WG method and CR method are presented in Section 6.
	
	\section{Notations and weak form}In this section, we state some notations and introduce the elastic eigenvalue problem. Let $\Omega\subset\Real^2$ be a bounded domain with $\partial\Omega=\Gamma_D\cup\Gamma_N$, and $\bh^m(\Omega)$ be the Sobolev spaces. The notations $(\cdot, \cdot)_{m,D}$, $||\cdot||_{m,D}$ and $|\cdot|_{m,D}$ are used as inner-product, norms and seminorms on $\bh^m(D)$, if the region $D$ is an edge of some elements, we use
$\langle\cdot,\cdot\rangle _{m,D}$ instead of $(\cdot,
\cdot)_{m,D}$. For simplicity, We shall drop the subscript when $m=0$ or
$D=\Omega$. Define $\bh_E^1(\Omega)=\{\bv\in\bh^1(\Omega):\bv|_{\Gamma_D}=\b0\}$.

In this paper, we consider the following linear elastic eigenvalue problem: Find $\gamma\in\Real$ and $\bu\in\bh_E^1(\Omega)$ such that
\begin{equation}\label{eig}
	\left\{
	\begin{array}{rcl}
		-\nabla\cdot \sigma(\bu) &=& \gamma \bu,\quad \text{in}\quad\Omega,\\
		\bu &=& 0,\quad~~ \text{on}\quad\Gamma_D,\\
		\sigma(\bu)\bn &=& 0,\quad~~ \text{on}\quad\Gamma_N,\\
		\int_\Omega \bu^2d\Omega &=& 1,
	\end{array}
	\right.
\end{equation}
where $|\Gamma_D|>0$, $\bn$ is the unit outward normal vector of $\Gamma_N$. The stress tensor $\sigma(\bu)$ is given by
\begin{eqnarray*}
	\sigma(\bu)=2\mu\eps(\bu)+\lambda(\nabla\cdot\bu)\bi,
\end{eqnarray*}
where $\bi\in\Real^{2\times 2}$ is the identity matrix. The strain tensor $\eps(\bu)$ is defined as
\begin{eqnarray*}
	 \eps(\bu)=\frac{1}{2}(\nabla\bu+(\nabla\bu)^T).
\end{eqnarray*}
The Lam\'{e} parameters $\mu$ and $\lambda$ are given by
\begin{eqnarray*}
	\lambda=\frac{E\nu}{(1+\nu)(1-2\nu)}\ \ \ \ \ {\rm and}\ \ \ \ \ \mu=\frac{E}{2(1+\nu)},
\end{eqnarray*}
where $E$ denotes the Young's modulus and $\nu\in(0,0.5)$ is the Poisson ratio.

The variational form of problem (\ref{eig}) is: Find $\gamma\in\Real$ and $\bu\in H_E^1(\Omega)$ such that $b(\bu,\bu)=1$ and
\begin{eqnarray}\label{vartion1}
	a(\bu,\bv)=\gamma b(\bu,\bv),\quad\forall\bv\in H_E^1(\Omega),
\end{eqnarray}
where
\begin{eqnarray*}
	a(\bu,\bv) &=& 2\mu(\eps(\bu),\eps(\bv))+\lambda(\di\bu,\di\bv),\\
	b(\bu,\bv) &=& (\bu,\bv).
\end{eqnarray*}
It is well known that problem \eqref{vartion1} has the
eigenvalue sequence \cite{MR1115240}
\begin{eqnarray*}
0<\gamma_1\le\gamma_2\le\ldots\le\gamma_j\leq\ldots\longrightarrow +\infty
\end{eqnarray*}
with the corresponding eigenfunction sequence
\begin{eqnarray*}
	\bu_1,\bu_2,\ldots,\bu_j,\ldots,
\end{eqnarray*}
such that $b(\bu_i,\bu_j)=\delta_{ij}$, $i,j=1,2,\cdots$.

\begin{remark}
	From \eqref{eig} and \eqref{vartion1}, $\lambda||\di\bu||^2$ and $\lambda||\di\bu||_1^2$ are both bounded, we further assume the following assumption holds.
\end{remark}

\textbf{Assumption (A0)}\emph{
	For k $\geq$ 1, $\lambda||\text{div}\bu||_k^2$ are bounded.
}
	
	\section{Abstract framework}
	In this section, we introduce an framework for WG method to solve eigenvalue problem, more detailed information can be found in \cite{Zhai2019}. Suppose $(V,(\cdot,\cdot)_V)$ is a Hilbert space, $V_c$ and $V_h$ are two
	subspaces of $V$. Let $a(\cdot,\cdot)$ be a bilinear form on $V_c\times V_c$,
	$a_w(\cdot,\cdot)$ be a bilinear form on $V_h\times V_h$, and $b(\cdot,\cdot)$ be a bilinear form on $V\times V$. We consider the eigenvalue problems:
	
	Find $(\gamma,u)\in \Real\times V_c$ and $(\gamma_h,u_h)\in\Real\times V_h$ such
	that $b(u,u)=b(u_h,u_h)=1$ and
	\begin{eqnarray}\label{problem-eq1}
		a(u,v) &=& \gamma b(u,v),\quad\forall v\in V_c,
		\\ \label{problem-eq2}
		a_w(u_h,v_h) &=& \gamma_h b(u_h,v_h),\quad\forall v_h\in V_h.
	\end{eqnarray}
	We assume that the bilinear forms $a$, $a_w$ and $b$ have the following property:
	
	\textbf{Assumption (A1)}\emph{
		$a$, $a_w$, and $b$ are symmetric, and
		for any $v\in V_c$ and $v_h\in V_h$,
		\begin{align*}
			a(v,v) \ge& \gamma_c\|v\|_V^2,
			\\
			a_w(v_h,v_h) \ge& \gamma(h)\|v_h\|_V^2,
		\end{align*}
		where $\gamma_c$ and $\gamma(h)$ are positive constant and function, repectively.
	}
	
	The eigenvalue problems \eqref{problem-eq1} and \eqref{problem-eq2} can be viewed as
	operator spectrum problems. Define two operators $K:V_c\rightarrow V_c$ and $K_h: V_h
	\rightarrow V_h$ satisfying
	\begin{align}
		a(Kf, v) =& b(f,v),\quad\forall v\in V_c,\label{vartion}
		\\
		a_w(K_h f_h,v_h) =& b(f_h,v_h),\quad\forall v_h\in V_h.
	\end{align}
	Thanks to the Lax-Milgram theory, it is easy to check $K$ and $K_h$ are well-defined by Assumption (A1).
	
	
	
	We denote by $\sigma(K)$ the spectrum of $K$, and by $\rho(K)$ the
	resolvent set. $R_z(K)=(zI-K)^{-1}$ represents the resolvent
	operator for any $z\in\rho(K)$. Let $\mu$ be a nonzero eigenvalue of $K$ with algebraic
	or geometric
	multiplicities $m$. Let $\Gamma_\mu$ be a circle in the complex plane
	centered at $\mu$ which lies in $\rho(K)$ and encloses no other
	points of $\sigma(K)$. The corresponding spectral projection is
	\begin{eqnarray*}
		E_\mu(K)=\frac{1}{2\pi {\rm i}}\int_{\Gamma_\mu} R_z(K) dz.
	\end{eqnarray*}
	It is known that the range $R(E_\mu(K))$ of $E_\mu(K)$ is the eigenspace corresponding to the eigenvalue $\mu$. Then we make the following assumption.

	\textbf{Assumption (A2)}\emph{ $K$ and $K_h$ are compact.}

	\textbf{Assumption (A3)}\emph{ There exists a bounded linear operator $Q_h:V\rightarrow V_h$ satisfying
		\begin{align*}
			Q_h v_h =& v_h,\quad\forall v_h\in V_h,
			\\
			b(Q_h w,v_h) =& b(w,v_h),\quad\forall w\in V,v_h\in V_h.
		\end{align*}
	}

	\textbf{Assumption (A4)}\emph{ For any $\mu\in\sigma(K)$, there holds
		\begin{eqnarray*}
			e_{h,\mu}\rightarrow 0\ and \ \delta_{h,\mu}\gamma(h)^{-1}\rightarrow 0\ as \ h\rightarrow 0,
		\end{eqnarray*}
		where $e_{h,\mu} = \|(K-K_hQ_h)|_{R(E_\mu(K))}\|_V$, $\delta_{h,\mu} = \sup\limits_{\substack{u\in R(E_\mu(K))\\ \|u\|_V=1}}\|u-Q_h u\|_V$.
	}

	Let $X$ and $Y$ are two subspaces of a Banach space $V$, the distance is defined by
	\begin{eqnarray}
		\rho_V(X,Y)=\sup_{\substack{x\in X\\\|x\|_V=1}} \rho(x,Y),~
		\hat \rho_V(X,Y)=\max\{\rho_V(X,Y),\rho_V(Y,X)\}.
	\end{eqnarray}
	
	Under Assumptions (A1)-(A4), we have the following result according to Theorem 7.1 in \cite{MR1115240}.
	\begin{theorem}\label{thm-vnorm}
		 For any $\mu\in\sigma(K)$, there exist a constant $C$ only depends on $K$ such that
		\begin{equation*}
			\hat \rho_V(R(E_\mu(K)), R(E_{\mu,h}(K_h)) \le C(e_{h,\mu}+\delta_h\gamma(h)^{-1}).
		\end{equation*}
	\end{theorem}
	
	Suppose $X$ is a Hilbert space equipped with inner-product $b(\cdot,\cdot)$, and $\|\cdot\|_X$ is the
	corresponding norm. In the elliptic eigenvalue problems, $X$ is $L^2(\Omega)$. Suppose $\Pi_0$ is
	the orthogonal projection from $V$ onto $X$ under $b(\cdot,\cdot)$.
	Then
	$\Pi_0K|_{\Pi_0V_c}: \Pi_0V_c\rightarrow \Pi_0V_c$ and $\Pi_0K_h|_{\Pi_0V_h}: \Pi_0V_h\rightarrow \Pi_0V_h$ are bounded linear operators.

	\textbf{Assumption (A5)}\emph{ For any $\mu\in\sigma(K)$, there holds
		\begin{eqnarray*}
			e'_{h,\mu}\rightarrow 0\ and \ \delta'_h\gamma(h)^{-1}\rightarrow 0 \ as \ h\rightarrow 0,
		\end{eqnarray*}
		where $e'_{h,\mu} = \|(\Pi_0 K-\Pi_0K_hQ_h)|_{R(E_\mu(K))}\|_X$ and $\delta'_h = \rho_X(V,V_h)$.}
	

	By the Assumption (A1)-(A5), we can derive the following estimate.
	\begin{theorem}\label{thm-xnorm}
		For any $\mu\in\sigma(K)$ there exist a constant $C$ depend on $\mu$ such that
		\begin{equation*}
			\hat \rho_X(R(E_\mu(K)), R(E_{\mu,h}(K_h)) \le C(e'_{h,\mu}+\delta'_h\gamma(h)^{-1}).
		\end{equation*}
	\end{theorem}
	
	Next we turn to the estimation of the eigenvalues.
	\begin{lemma}\label{lemma-expansionQhu}
		Under Assumption (A1) and (A3), suppose $(\gamma,u)$ is the solution of \eqref{problem-eq1} and $(\gamma_h,u_h)$ is the
		solution of \eqref{problem-eq2}. Then for any $v_h\in V_h$ we have the following expansion
		\begin{align*}
			\gamma-\gamma_h &= a(u,u) - a_w(Q_hu,Q_hu) + a_w(u_h-Q_hu,u_h-Q_hu)
			\\
			&\quad-\gamma_h b(u-u_h,u-u_h).
		\end{align*}
	\end{lemma}

	\textbf{Assumption (A6)}\emph{ For any $\mu\in \sigma(K)$ and $u\in R(E_\mu(K))$, there holds
		\begin{eqnarray*}
			\varepsilon_{h,u}\rightarrow 0\ as \ h\rightarrow 0,
		\end{eqnarray*}
		where $\varepsilon_{h,u} = a(u,u) - a_w(Q_hu,Q_hu)$.}
	
	With the assumptions above we are able to demonstrate the error estimate for eigenvalues.
	\begin{theorem}\label{thm-eigenvalue}
		Suppose $\gamma$ is an eigenvalue of \eqref{problem-eq1}, and $\{u_j\}_{j=1}^m$ are the corresponding eigenfunctions. Then when $h$ is  small enough there exists $m$ eigenvalues of
		\eqref{problem-eq2} $\{\gamma_{h,j}\}_{j=1}^m$ such that
		\begin{align*}
			|\gamma-\gamma_{h,j}|\le& C(\varepsilon_{h,u_j} + e_{h,\gamma^{-1}}^2 + e^{\prime 2}_{h,\gamma^{-1}}+ \delta_h^2\gamma(h)^{-2} + \delta^{\prime 2}_h\gamma(h)^{-2}) .
		\end{align*}
	\end{theorem}

	\textbf{Assumption (A7)}\emph{ Suppose $(\gamma,u)$ and $(\gamma_h,u_h)$ is an eigenpair of \eqref{problem-eq1} and \eqref{problem-eq2}, respectively. There holds
		\begin{eqnarray*}
			\varepsilon_{h,u} \ge \gamma_h\|u-u_h\|_X^2.
		\end{eqnarray*}
	}
	
	\begin{theorem}\label{thm-lower}
		Suppose $(\gamma,u)$ and $(\gamma_h,u_h)$ is an eigenpair of \eqref{problem-eq1} and \eqref{problem-eq2}, respectively. Then if Assumption (A1), (A3) and (A7) hold, we have
		\begin{eqnarray*}
			\gamma \ge \gamma_h.
		\end{eqnarray*}
	\end{theorem}

	\section{Application to linear elastic eigenvalue problems}
In this section, we apply the framework in Section 3 to solve problem \eqref{eig} by the weak Galerkin method.

We start by introducing some notations in the WG scheme.
Let $\T_h$ be a partition of the domain $\Omega$, and the elements
in $\T_h$ are polygons satisfying the regular assumptions specified
in \cite{Wang2014a}. Let $\E_h=\E_h^0\cup\E_h^b$ be the edges in $\T_h$, where
$\E_h^0$ and $\E_h^b$ denotes by the interior edges and boundary edges, respectively. For each element $T\in\T_h$, $h_T$ represents the diameter of $T$, and $h=\max\limits_{T\in\T_h} h_T$ denotes the mesh size. For simplicity, we use $a\lesssim b$ and $a\gtrsim b$ instead of $a\le Cb$ and $a\ge Cb$, respectively, where $C$ is a positive constant that is independent of $h$ and $\lambda$.

Now we introduce the WG scheme to solve problem (\ref{eig}).
For a given integer $k\ge 1$, define the WG finite element space
\begin{eqnarray*}
	V_h=\big\{\bv=\{\bv_0,\bv_b\}:\bv_0|_T\in \bP_k(T), \bv_b|_e\in \bP_k(e), \forall T\in\T_h, e\in\E_h,\text{ and } \bv_b=\b0 \text{ on }\Gamma_D\big\},
\end{eqnarray*}
where $\bP_k(T)=[P_k(T)]^2$ denotes the vectorial polynomail space on $T$ with degree no more than $k$, and
$\bP_k(e)=[P_k(e)]^2$ denotes the vectorial polynomail space on $e$ with degree no more than $k$.

Refer to \cite{Wang2013a}, the following trace inequality and inverse inequality hold true.
\begin{lemma}
	For any element  $T\in\T_h$, there holds the trace inequality
	\begin{eqnarray*}
		\|\bv\|_{\partial T}^2\lesssim h_T^{-1}\|\bv\|_T^2+h_T\|\nabla\bv\|_T^2,\quad\forall\bv\in\bh^1(T).
	\end{eqnarray*}
\end{lemma}
\begin{lemma}
	For any element $T\in\T_h$, there holds the inverse inequality
	\begin{eqnarray*}
		\|\nabla\bv\|_T\lesssim h_T^{-1}\|\bv\|_T,\quad\forall\bv\in\bP_k(T).
	\end{eqnarray*}
\end{lemma}

For the aim of analysis, some projection operators are also employed in this paper. For each $T\in\T_h$, let $Q_0$
denotes the $L^2$ projection from $\bl^2(T)$ onto $\bP_k(T)$, 
$\dQ_h$
denotes the $L^2$ projection from $[L^2(T)]^{2\times2}$ onto $[P_{k-1}(T)]^{2\times 2}$,
and $\Q_h$ denotes the $L^2$ projection from $L^2(T)$ onto $P_{k-1}(T)$.
For each $e\in\E_h$, let $Q_b$ denotes the $L^2$ projection from $\bl^2(e)$ onto $\bP_k(e)$ for each $e\in\E_h$.
Combining $Q_0$ and $Q_b$ together, we define $Q_h=\{Q_0,Q_b\}$,
which is a projection onto $V_h$.

Refer to \cite{Mu2013c} and \cite{Wang2015a}, the following estimates hold true for the projection operators.
\begin{lemma}\label{projection}
	For any $\bv\in\bh^{k+1}(\Omega)$ and $\tau\in H^k(\Omega)$, $s=0,1$, we have
	\begin{align*}
		\sumT h^{2s}\|\bv-Q_0\bv\|_{s,T}^2&\lesssim h^{2(k+1)}\|\bw\|_{k+1}^2,\\
		\sumT h^{2s}\|\nabla\bv-\dQ_h\nabla\bv\|_{s,T}^2&\lesssim  h^{2k}\|\bw\|_{k+1}^2,\\
		\sumT h^{2s}\|\eps(\bv)-\dQ_h(\eps(\bv))\|_{s,T}^2&\lesssim  h^{2k}\|\bw\|_{k+1}^2,\\
		\sumT h^{2s}\|\tau-\Q_h\tau\|_{s,T}^2&\lesssim h^{2k}\|\tau\|_k^2.
	\end{align*}
\end{lemma}

For each $\bv\in V_h+\bh^1(\Omega)$, we define its weak gradient
$\nabla_w\bv$ and weak strain tensor $\eps_w(\bv)$ as follows.
\begin{definition}
	$\nabla_w\bv|_T$ is the unique polynomial in $[P_{k-1}(T)]^{2\times 2}$ satisfying
	\begin{eqnarray}\label{defw1}
		(\nabla_w\bv,q)_T=-(\bv_0,\nabla\cdot q)_T+\langle \bv_b,q\bn
		\rangle_{\partial T},\quad\forall q\in [P_{k-1}(T)]^{2\times 2},
	\end{eqnarray}
	where $\bn$ denotes the outward unit normal vector and define
	\begin{eqnarray}\label{defw2}
		\eps_w(\bv)=\frac{1}{2}(\nabla_w\bv+(\nabla_w\bv)^T).
	\end{eqnarray}
\end{definition}

For each $\bv\in V_h+\bh^1(\Omega)$, we define its weak divergence
$\nabla_w\cdot\bv$ as follows.
\begin{definition}
	$\nabla_w\cdot v|_T$ is the unique polynomial in $P_{k-1}(T)$ satisfying
	\begin{eqnarray}\label{defw3}
		(\nabla_w\cdot\bv,\tau)_T=-(\bv_0,\nabla\tau)_T+\langle \bv_b\cdot\bn,\tau
		\rangle_{\partial T},\quad\forall\tau\in P_{k-1}(T),
	\end{eqnarray}
	where $\bn$ denotes the outward unit normal vector.
\end{definition}

The following commutative properties for the weak differential operator
plays an essential role in the analysis. The proof can be found in Lemma 6.1 \cite{Wang2016c}.
\begin{lemma}\label{exchange}
	For any $\bv\in\bh^1(\Omega)$, there holds that on each element $T\in\T_h$
	\begin{align}
		\nabla_w (Q_h\bv) =& \dQ_h\nabla\bv,\label{exchange1}\\
		\eps_w (Q_h\bv) =& \dQ_h(\eps(\bv)),\label{exchange2}\\
		\nabla_w\cdot (Q_h\bv) =& \Q_h(\nabla\cdot\bv).\label{exchange3}
	\end{align}
\end{lemma}

Furthermore, silimar to the property in Lemma 2.1 in \cite{Ye2020}, we have the following properties.
\begin{lemma}
	For any $\bv\in\bh^1(\Omega)$, there holds that on each element $T\in\T_h$
	\begin{align}
		\nabla_w\bv =& \nabla_w(Q_h\bv),\label{property1}\\
		\nabla_w\cdot\bv =& \nabla_w\cdot(Q_h\bv).\label{property2}
	\end{align}
\end{lemma}
\begin{proof}
	Refer to Lemma 2.1 in \cite{Ye2020}, we have
	\begin{align}\label{exchange4}
		\nabla_w\bv =& \dQ_h\nabla\bv,
	\end{align}
	which combines with \eqref{exchange1} verifies \eqref{property1}. Then it remains to prove \eqref{property2}.
	
	By the definition in \eqref{defw3} and Green formulation, for any $\tau\in P_{k-1}(T)$, we have
	\begin{align*}
		(\nabla_w\cdot\bv,\tau)_T
		&=-(\bv,\nabla\tau)_T+\la\bv\cdot\bn,\tau
		\ra\\
		&=-(Q_0\bv,\nabla\tau)_T+\la Q_b\bv\cdot\bn,\tau
		\ra\\
		&=(\nabla_w\cdot(Q_h\bv),\tau)_T,
	\end{align*}
	which verifies \eqref{property2}. The proof is complete.
\end{proof}

Next we define two bilinear forms on $V_h$. For any $\bv,\bw\in V_h$,
\begin{align*}
	s(\bv,\bw)=&\gamma(h)\sumT h_T^{-1}\langle \bv_0-\bv_b, \bw_0-\bw_b\rangle_{\partial T},\\
	b_w(\bv,\bw)=&(\bv_0,\bw_0),
\end{align*}
where $\gamma(h)= h^{\delta}$ for small positive constant $\delta$.

Then we are ready to state the WG algorithm.
\begin{algorithm1}
	Find $\bu_h\in V_h$, $\gamma_h\in\Real$ such that $b_w(\bu_h,\bu_h)=1$ and
	\begin{eqnarray}\label{WGscheme}
		a_w(\bu_h,\bv)=\gamma_h b_w(\bu_h,\bv),\quad\forall \bv\in V_h,
	\end{eqnarray}
	where
	\begin{eqnarray*}
		a_w(\bv,\bw)=2\mu(\eps_w(\bv),\eps_w(\bw))+\lambda(\nabla_w\cdot \bv,\nabla_w\cdot\bw)+s(\bv,\bw).
	\end{eqnarray*}
\end{algorithm1}

Next, we verify Assumptions (A1)-(A6) in Section 3 to give the
error estimation for the weak Galerkin algorithm.

Denote $V_c=\bh_E^1(\Omega)$,
and $V =V_c+V_h$. For any $\bv,\bw\in V$, the inner-product is given by
\begin{equation*}
	(\bv,\bw)_V = (\eps(\bv_0),\eps(\bw_0)) +\lambda(\nabla_w\cdot\bv,\nabla_w\cdot\bw)+\sumT h_T^{-1}\la \bv_0-\bv_b,\bw_0-\bw_b\ra,
\end{equation*}
and the corresponding semi-norm is
\begin{equation*}
	\|\bv\|_V^2 = \sumT\|\eps(\bv_0)\|_T^2 + \lambda\sumT\|\nabla_w\cdot\bv\|_T^2 +\sumT h_T^{-1}\|\bv_0-\bv_b\|_{\partial T}^2.
\end{equation*}
By the Korn's inequality \cite{Wang2015a}, it is easy to check that $\|\cdot\|_V$ defines a norm on $V$, which means $V$ is a Hilbert space with the norm $\|\cdot\|_V$.

First we verify Assumption (A1). It is easy to check that $a(\cdot,\cdot)$ and $a_w(\cdot,\cdot)$ defined in \eqref{vartion1} and \eqref{WGscheme} are symmetric bounded bilinear forms on $V$, and $a(\cdot,\cdot)$ is positive definite. Hence we only need to verify the coercivity of $a_w(\cdot,\cdot)$.
\begin{lemma}\label{coe}
	For any $\bv_h\in V_h$, the following inequality holds
	\begin{eqnarray*}
		a_w(\bv_h,\bv_h)\gtrsim\gamma(h)\|\bv_h\|_V^2.
	\end{eqnarray*}
\end{lemma}
\begin{proof}
	From the definition \eqref{defw1} and \eqref{defw2}, we have
	\begin{align*}
		\sumT(\eps(\bv_0),\eps(\bv_0))_T=&\sumT(\nabla \bv_0,\eps(\bv_0))_T
		\\
		=& -\sumT(\bv_0,\di(\eps(\bv_0)))_T + \sumT\la\bv_0,\eps(\bv_0)\bn\ra
		\\
		=& \sumT(\nabla_w \bv_h,\eps(\bv_0))_T + \sumT\la \bv_0-\bv_b,\eps(\bv_0)\bn\ra
		\\
		=& \sumT(\eps_w(\bv_h),\eps(\bv_0))_T + \sumT\la \bv_0-\bv_b,\eps(\bv_0)\bn\ra.
	\end{align*}
	Then, by the trace inequality and the inverse inequality we derive
	\begin{align*}
		\sumT\|\eps(\bv_0)\|_T^2
		\lesssim & \left(\sumT\|\eps_w(\bv_h)\|_T^2\right)^\frac12\left(\sumT\|\eps(\bv_0)\|_T^2\right)^\frac12
		\\
		\quad&+\left(\sumT h_T^{-1}\|\bv_0-\bv_b\|_{\partial T}^2\right)^\frac12\left(\sumT h_T\|\eps(\bv_0)\bn\|_{\partial T}^2\right)^\frac12
		\\
		\lesssim & \left(\sumT\|\eps_w(\bv_h)\|_T^2\right)^\frac12\left(\sumT\|\eps(\bv_0)\|_T^2\right)^\frac12
		\\
		\quad&+\left(\sumT h_T^{-1}\|\bv_0-\bv_b\|_{\partial T}^2\right)^\frac12\left(\sumT \|\eps(\bv_0)\|_T^2\right)^\frac12
		\\
		\lesssim & \left(\sumT\|\eps_w(\bv_h)\|_T^2+\sumT h_T^{-1}\|\bv_0-\bv_b\|_{\partial T}^2\right)^\frac12\left(\sumT \|\eps(\bv_0)\|_T^2\right)^\frac12.
	\end{align*}
	Therefore, we acquire
	\begin{eqnarray}\label{e1}
		\sumT\|\eps(\bv_0)\|_T^2
		\lesssim \gamma(h)^{-1}a_w(\bv_h,\bv_h).
	\end{eqnarray}
	Moreover, it is straightforward to deduce
	 \begin{eqnarray}\label{e2}
	 	\sumT h_T^{-1}\|\bv_0-\bv_b\|_{\partial T}^2\lesssim \gamma(h)^{-1}a_w(\bv_h,\bv_h).
	 \end{eqnarray} 
	Combine \eqref{e1} with \eqref{e2} yields
	\begin{eqnarray*}
		\|\bv_h\|_V^2\lesssim \gamma(h)^{-1}a_w(\bv_h,\bv_h).
	\end{eqnarray*}
	The proof is complete.
\end{proof}

Then we turn to Assumption (A2). Recall that in Section 3 the operators
$K$ and $K_h$ are defined by
\begin{align*}
	a(K\bf,\bv)=& b(\bf,\bv),\quad\forall\bf\in V_c,
	\\
	a_w(K_h\bf_h,\bv_h)=& b(\bf_h,\bv_h),\quad\forall \bf_h\in V_h.
\end{align*}

From Lemma 2.1 in \cite{Mora2020} and the compact inclusion $\bh^{1+r}(\Omega)\hookrightarrow\bh^1(\Omega)$, we get the operator $K$ is compact. As to $K_h$, notice that $K_h$ is a bounded linear and finite ranked operator. Thus $K_h$ is also compact, which verifies Assumption (A2).

Recall the operator $Q_h=\{Q_0,Q_b\}$, where $Q_0$ is the $L^2$ projection operator
onto $\bP_k(T)$ on each element $T\in\T_h$ and $Q_b$ the $L^2$ projection operator
onto $\bP_k(e)$ on each edge $e\in\E_h$. We claim that this operator $Q_h$ satisfies Assumption (A3).
It is obvious that $Q_h\bv_h=\bv_h$ for any $\bv_h\in V_h$, and by the usual property of projection $Q_0$, it follows that for any $\bw\in V$
\begin{eqnarray*}
	b(\bw,\bv_h) = (\bw_0,\bv_0) = (Q_0\bw_0,\bv_0) = b(Q_h\bw,\bv_h).
\end{eqnarray*}
Thus, Assumption (A3) is verified.

In order to obtain the estimation on the eigenfunctions, we need to verify Assumptions (A4) and (A5).
\begin{lemma}\label{delta}
	Suppose $R(E_\mu(K))\subset \bh^k(\Omega)$, then the following estimations hold
	\begin{align*}
		\delta_{h,\mu} \lesssim & h^k,
		\\
		\delta'_{h,\mu} \lesssim & h^{k+1}.
	\end{align*}
\end{lemma}
\begin{proof}
	By Lemma \ref{projection} we have
	\begin{eqnarray*}
		\|\bv-Q_h\bv\|_X = \sumT\|\bv-Q_0\bv\|_T\lesssim  h^{k+1}.
	\end{eqnarray*}	
	Furthermore, by Lemma \ref{projection} and Lemma \ref{exchange}, we derive
	\begin{align*}
		\|\bv-Q_h\bv\|_V^2 =& 
		\sumT\|\eps(\bv-Q_0\bv)\|_T^2 + \lambda\sumT\|\nabla_w\cdot (\bv-Q_h\bv)\|_T^2\\
		\quad&+ \sumT h_T^{-1}\|Q_0\bv- Q_b\bv\|_{\partial T}^2
		\\
		\le& \sumT\|\nabla(\bv-Q_0\bv)\|_T^2 + \sumT h_T^{-1}\|Q_0\bv- \bv\|_{\partial T}^2
		\\
		\le& \sumT\|\nabla(\bv-Q_0\bv)\|_T^2 + \sumT h_T^{-2}\|Q_0\bv- \bv\|_T^2
		\\
		\lesssim & h^{2k}.
	\end{align*}
	The proof is complete.
\end{proof}

Denote
\begin{align*}
	e_{h,\mu} =& \|( K-K_hQ_h)|_{R(E_\mu(K))}\|_V,
	\\
	e'_{h,\mu} =& \|(\Pi_0 K-\Pi_0K_hQ_h)|_{R(E_\mu(K))}\|_X.
\end{align*}

Since $e_{h,\mu}$ and $e'_{h,\mu}$ are errors of the WG method for the linear elasticity equation, we consider the source problem of \eqref{vartion1}
\begin{eqnarray}\label{source}
	a(\bu,\bv)=b(\bf,\bv),\quad\forall\bv\in\bh_E^1(\Omega),
\end{eqnarray}
with $\bf\in\bl^2(\Omega)$ is a given function. The corresponding WG scheme
is to find $\bu_h\in V_h$ such that
\begin{eqnarray}\label{sourceh}
	a_w(\bu_h,\bv_h)=b_w(\bf,\bv_0),\quad \forall\bv_h\in V_h.
\end{eqnarray}

The error estimates of \eqref{sourceh} are analyzed in Appendix (See Theorem \ref{Vnorm} and Theorem \ref{Xnorm}). Then the estimations for $e_{h,\mu}$ and $e'_{h,\mu}$ are as follows.
\begin{lemma}\label{e}
	Suppose $R(E_\mu(K))\subset\bh^{k+1}(\Omega)$, then the following estimations hold
	\begin{align*}
		e_{h,\mu} \lesssim & \gamma(h)^{-1}h^k,\\
		e'_{h,\mu} \lesssim & \gamma(h)^{-1}h^{k+1}.
	\end{align*}
\end{lemma}

Thus, according to Theorem \ref{thm-vnorm} and Theorem \ref{thm-xnorm}, we have the
estimations on the eigenfunctions.
\begin{theorem}\label{eigenvector}
	Suppose $\gamma$ is an eigenvalue of \eqref{vartion1} with multiplicity $m$,
	and $R(E_\mu(K))\subset\bh^{k+1}(\Omega)$ is the corresponding $m$-dimensional eigenspace. Suppose \\ $\{\gamma_{h,j}\}_{j=1}^m$
	are the eigenvalues of \eqref{WGscheme} approximating $\gamma$, and $\{\bu_{h,j}\}_{j=1}^m$ is
	a basis of the corresponding eigenspace $R(E_{\mu,h}(K_h))$. Then, for any $j=1,\cdots,m$ there
	exists an eigenfunction $\bu_j\in R(E_\mu(K))$ such that
	\begin{align*}
		\|\bu_j-\bu_{j,h}\|_V \lesssim & \gamma(h)^{-1}h^k,
		\\
		\|\bu_j-\bu_{j,h}\|_X \lesssim & \gamma(h)^{-1}h^{k+1}.
	\end{align*}
\end{theorem}

To estimate the errors of eigenvalues, we have the following result, which verifies Assumption (A6).
\begin{lemma}\label{eps}
	For any $\bu\in \bh^{k+1}(\Omega)$, the following estimate holds
	\begin{eqnarray}
		|\varepsilon_{h,u}| \lesssim h^{2k}.
	\end{eqnarray}
\end{lemma}
\begin{proof}
	It follows from Lemma \ref{exchange} and the trace inequality that
	\begin{align*}
		|\varepsilon_{h,u}| =& |a(\bu,\bu) - a_w(Q_h\bu,Q_h\bu)|
		\\
		=& \left|2\mu(\|\eps(\bu)\|^2 - \sumT\|\dQ_h(\eps(\bu))\|_T^2)+\lambda(\|\di\bu\|^2 - \sumT\|\mathcal{Q}_h(\di\bu)\|_T^2) -s(Q_h\bu,Q_h\bu)\right|
		\\
		=& 2\mu\sumT\|\eps(\bu)-\dQ_h(\eps(\bu))\|_T^2 +\lambda\sumT\|\di \bu-\mathcal{Q}_h(\di\bu)\|_T^2
		\\
		\quad&+ \gamma(h)\sumT h_T^{-1}\|Q_0\bu -Q_b\bu)\|_{\partial T}^2
		\\
		\le& 2\mu\sumT\|\eps(\bu)-\dQ_h(\eps(\bu))\|_T^2 +\lambda\sumT\|\di \bu-\mathcal{Q}_h(\di\bu)\|_T^2
		\\
		\quad&+ \gamma(h)\sumT h_T^{-1}\|Q_0\bu -\bu\|_{\partial T}^2
		\\
		\lesssim & h^{2k},
	\end{align*}
	which completes the proof.
\end{proof}

By Theorem \ref{thm-eigenvalue} we derive the estimation on the eigenvalues.
\begin{theorem}\label{eigenvalue}
	Suppose $\gamma$ is an eigenvalue of \eqref{vartion1} with multiplicity $m$,
	and $R(E_\mu(K))\subset \bh^{k+1}(\Omega)$ is corresponding $m$-dimensional eigenspace. Suppose $\{\gamma_{h,j}\}_{j=1}^m$
	are the eigenvalues of \eqref{WGscheme} approximating $\gamma$, and $\{\bu_{h,j}\}_{j=1}^m$ is
	a basis of the corresponding eigenspace $R(E_{\mu,h}(K_h))$. Then when $h$ is  small enough, for any $j=1,\cdots,m$ there
	holds
	\begin{eqnarray}\label{gammaerr}
		|\gamma-\gamma_{h,j}|\lesssim \gamma(h)^{-2}h^{2k}.
	\end{eqnarray}
\end{theorem}
\begin{proof}
	From Theorem \ref{thm-eigenvalue} we have
	\begin{eqnarray*}
		|\gamma-\gamma_{h,j}|\lesssim \varepsilon_{h,u_j} + e_{h,\gamma^{-1}}^2 + e^{\prime 2}_{h,\gamma^{-1}}+ \gamma(h)^{-2} \delta_{h,\gamma^{-1}}^2+ \gamma(h)^{-2} \delta^{\prime 2}_{h,\gamma^{-1}}.
	\end{eqnarray*}
	It follows from Lemma \ref{delta}, Lemma \ref{e}, and Lemma \ref{eps} that
	\begin{align*}
		\varepsilon_{h,u_j} \lesssim & h^{2k},\quad
		e_{h,\gamma^{-1}}^2 \lesssim \gamma(h)^{-2}h^{2k},
		\\
		e'^2_{h,\gamma^{-1}} \lesssim & \gamma(h)^{-2}h^{2k+2},\quad
		\delta_{h,\gamma^{-1}}^2\lesssim h^{2k},
		\\
		\delta'^2_{h,\gamma^{-1}} \lesssim & h^{2k+2},
	\end{align*}
	which implies \eqref{gammaerr} and finishes the proof.
\end{proof}

Next, we prove the WG scheme \eqref{WGscheme} provides asymptotic lower bounds of the eigenvalues by verifying Assumption (A7). The following lower bound estimate is crucial in the analysis, which is proved in Theorem 2.1 \cite{Lin2011a}.
\begin{lemma}\label{lowerbound}
	The following lower bound for the convergence rate holds for the exact eigenfunction $\bu$ of the eigenvalue problem \eqref{vartion1}
	\begin{eqnarray*}
		\sumT\|\eps(\bu)-\dQ_h(\eps(\bu))\|_T^2\gtrsim h^{2k}.
	\end{eqnarray*}
\end{lemma}

It remains to verify Assumption (A7).
\begin{lemma}\label{A72}
	Suppose $(\gamma,\bu)$ is an eigenpair of \eqref{vartion1} and $(\gamma_h,\bu_h)$ is an eigenpair of
	\eqref{WGscheme} approximating $(\gamma,\bu)$. Suppose $\gamma(h)\ll 1$, then when $h$ small enough, there has
	\begin{eqnarray*}
		\varepsilon_{h,u} \gtrsim \gamma_h\|\bu-\bu_h\|_X^2.
	\end{eqnarray*}
\end{lemma}
\begin{proof}
	From Lemma \ref{eps}, we have
	\begin{align*}
		\varepsilon_{h,u} =& |a(\bu,\bu)-a_w(Q_h\bu,Q_h\bu)|
		\\
		=& 2\mu\sumT\|\eps(\bu)-\dQ_h(\eps(\bu))\|_T^2+\lambda\sumT\|\di \bu-\Q_h(\di\bu)\|_T^2 \\
		\quad&-\gamma(h)\sumT h_T^{-1}\|Q_0\bu-Q_b\bu\|_{\partial T}^2\\
		\ge& 2\mu\sumT\|\eps(\bu)-\dQ_h(\eps(\bu))\|_T^2 
		-\gamma(h)\sumT h_T^{-1}\|\bu-Q_0\bu\|_{\partial T}^2.
	\end{align*}
	It follows from Lemma \ref{lowerbound} and the trace inequality that
	\begin{align*}
		\sumT\|\eps(\bu)-\dQ_h(\eps(\bu)\|_T^2 \gtrsim& h^{2k},
		\\
		\gamma(h)\sumT h_T^{-1}\|\bu-Q_0\bu\|_{\partial T}^2 \lesssim& \gamma(h)h^{2k}.
	\end{align*}
	Since $\gamma(h)\ll 1$, then when $h$ is sufficiently small we have
	\[
	\varepsilon_{h,u} \gtrsim h^{2k}.
	\]
	From Theorem \ref{eigenvector} we obtain
	\begin{eqnarray*}
		\gamma_h\|\bu-\bu_h\|_X^2\lesssim h^{2k+2}.
	\end{eqnarray*}
	Thus, when $h$ is small enough we derive
	\[
	\varepsilon_{h,u} \gtrsim \gamma_h\|\bu-\bu_h\|_X^2,
	\]
	which completes the proof.
\end{proof}

Applying Theorem \ref{thm-lower} we obtain the following lower bound estimation.
\begin{theorem}
	Suppose $(\gamma,\bu)$ is an eigenpair of \eqref{vartion1} and $(\gamma_h,\bu_h)$ is an eigenpair of
	\eqref{WGscheme} approximating $(\gamma,\bu)$, then when h is small enough, we have
	\begin{eqnarray*}
		\gamma \ge \gamma_h.
	\end{eqnarray*}
\end{theorem}

	\section{Lower bound for Crouzeix-Raviart element}In this section, we claim that the Crouzeix-Raviart element can also give asymptotic lower bounds for eigenvalues when the corresponding eigenfunctions are sigular. We still use the notations defined in Section 3. For any edge $e\in\E_h^0$, let $\bn_e$ be the unit normal of $e$ pointing from $T^{+}$ to $T^{-}$, for any function $\bv$, the jump $[\![\bv]\!]$ through $e$ is defined by $[\![\bv]\!]|_e=(\bv|_{T^{+}})|_e-(\bv|_{T^{-}})|_e$. For any $e\in\E_h^b$, we define $[\![\bv]\!]|_e=(\bv|_{T})|_e$. The finite element space $V_h$ is the corresponding finite elemnent space on the partition, and denote $V=V_h+\bh_E^1(\Omega)$.

The Crouzeix-Raviart element space is given by:
\begin{align*}
	V_h=&\left\{\bv\in\bl^2(\Omega):\bv|_T\in \bP_1(T),\int_{e}[\![\bv]\!]ds=\b0,\forall e\in\E_h^0\right.,\\
	&\quad\left. and~\int_e\bv ds=\b0,\forall e\in\E_h^b\cap\Gamma_D\right\}.
\end{align*}
Define the interpolation operator $\bi_h:\bh^1_E(\Omega)\rightarrow V_h$ by
\begin{eqnarray}\label{CR-int}
	\int_{e}(\bv-\bi_h\bv)ds=\b0,\quad\forall e\in\E_h.
\end{eqnarray}

The following error estimate for the interpolation operator can be found in \cite{Luo2012,Zhang2023}.
\begin{lemma}
	For any $\bv\in\bh^{1+s}(\Omega)$, there holds
\begin{align}\label{errint}
	\|\bv-\bi_h\bv\|+h|\bv-\bi_h\bv|_1\lesssim& h^{1+s}|\bv|_{1+s}.
\end{align}
\end{lemma}



Now we introduce the nonconforming CR finite element algorithm, which is based on \cite{Zhang2023}.
\begin{algorithm}Find $\gamma_h\in\Real$ and $\bu_h\in V_h$ such that $b_h(\bu_h,\bu_h)$\\$=1$ and
	\begin{eqnarray}\label{vartionh}
		a_h(\bu_h,\bv_h)=\gamma_hb_h(\bu_h,\bv_h),\quad\forall\bv_h\in V_h,
	\end{eqnarray}
\end{algorithm}
where
\begin{align*}
	a_h(\bu_h,\bv_h)=&2\mu\sumT(\eps(\bu_h),\eps(\bv_h))_T+\lambda\sumT(\di\bu_h,\di\bv_h)_T+\gamma(h)\sum\limits_{e\in\E_h^0}\frac{2\mu}{h_e}\la[\![\bu_h]\!],[\![\bv_h]\!]\rangle_e,\\
	b_h(\bu_h,\bv_h)=&\sumT(\bu_h,\bv_h)_T.
\end{align*}

For any $\bv_h\in V_h$, its norm $||\cdot||_h$ is defined by
\begin{align*}
	||\bv_h||_h^2=&2\mu\sumT||\eps(\bv_h)||_T^2+\lambda\sumT||\di\bv_h||_T^2+\sum\limits_{e\in\E_h^0}||{h_e}^{-\frac12}[\![\bv_h]\!]||_e^2,
\end{align*}
it is easy to know that $a_h(\cdot,\cdot)$ is continuous on $V_h$ and satisfies
\begin{eqnarray}\label{coeCR}
	a_h(\bv_h,\bv_h)&\gtrsim&\gamma(h)||\bv_h||_h^2,\quad\forall\bv_h\in V_h.
\end{eqnarray}

The following discrete Korn's inequality is proved in \cite{Korn2004} and \cite{Korn2000}.
\begin{lemma}For any $\bv\in V_h$, there holds
	\begin{align*}
		\sumT||\nabla\bv_h||_T^2\lesssim& \sumT||\eps(\bv_h)||_T^2+\sum\limits_{e\in\E_h^0}||h_e^{-{\frac{1}{2}}}[\![\bv_h]\!]||_e^2,
	\end{align*}
	which implies
	\begin{align*}
		\sumT|\bv_h|_{1,T}\lesssim& ||\bv_h||_h.
	\end{align*}
\end{lemma}

Using the same arguments in \cite{Zhang2023} and together with \eqref{coeCR} we have the following estimates.
\begin{lemma}
	Let $(\gamma,\bu)$ and $(\gamma_h,\bu_h)$ be the $j-$th eigenpair of \eqref{vartion1} and \eqref{vartionh}, respectively, then
	\begin{align}
		|\gamma-\gamma_h|\lesssim&\gamma(h)^{-2}h^{2s},
	\end{align}
	and there exists eigenfunction $\bu$ corresponding to $\gamma$ satisfies
	\begin{align}
		||\bu-\bu_h||_h\lesssim& \gamma(h)^{-1}h^s,\label{errH1}\\
		||\bu-\bu_h||\lesssim& \gamma(h)^{-2}h^{2s},\label{errL2}.
	\end{align}
\end{lemma}

Next, we consider the lower bound property for the CR finite element method. The following result is a combination of Lemma 2.1 in \cite{Luo2012} and \eqref{coeCR}.
\begin{theorem}
	Let $(\gamma,\bu)$ and $(\gamma_h,\bu_h)$ be the $j-$th eigenpair of \eqref{vartion1} and \eqref{vartionh}, respectively. Then
	\begin{align}\label{expand}
		\gamma-\gamma_h\ge& \gamma(h)||\bu-\bu_h||_h^2-\gamma_h||\bi_h\bu-\bu_h||^2\\
		&+~\gamma_h(||\bi_h\bu||^2-||\bu||^2)+2a_h(\bu-\bi_h\bu,\bu_h).\nonumber
	\end{align}
\end{theorem}

For the last term in \eqref{expand}, we have the following estimate.
\begin{theorem}
	Suppose that $\bu\in\bh^{1+s}(\Omega)$, we have
	\begin{align}\label{lb1}
		|a_h(\bu-\bi_h\bu,\bu_h)|\lesssim&h^{2s}.
	\end{align}
\end{theorem}
\begin{proof} It follows from the definition in \eqref{vartionh} that
	\begin{align*}
		a_h(\bu-\bi_h\bu,\bu_h)=&2\mu\sumT(\eps(\bu-\bi_h\bu),\eps(\bu_h))_T +\lambda\sumT(\di(\bu-\bi_h\bu),\di\bu_h)_T\\
		\quad&+\gamma(h)\sum\limits_{e\in\E_h^0}\frac{2\mu}{h_e}\langle[\![\bu-\bi_h\bu]\!],[\![\bu_h]\!]\rangle_e.
	\end{align*}
	By the Green formulation, we have
	\begin{align*}
		(\eps(\bu-\bi_h\bu),\eps(\bu_h))_T=&(\nabla(\bv-\bi_h\bv),\eps(\bu_h))_T
		\\
		=&\la\bv-\bi_h\bv,\eps(\bu_h)\bn\ra-(\bv-\bi_h\bv,\di(\eps(\bu_h)))_T,
	\end{align*}
	notice that $\eps(\bu_h)$ is a constant function, which together with \eqref{CR-int} verifies
	\begin{align*}
		(\eps(\bu-\bi_h\bu),\eps(\bu_h))_T=&0,\quad\forall T\in\T_h.
	\end{align*}
	Similarly, we get
	\begin{align*}
		(\di(\bu-\bi_h\bu),\di(\bu_h))_T=&0,\quad\forall T\in\T_h.
	\end{align*}
	Therefore, we have
	\begin{eqnarray}\label{ahM}
		a_h(\bu-\bi_h\bu,\bu_h)=\gamma(h)\sum\limits_{e\in\E_h^0}\frac{2\mu}{h_e}\langle[\![\bu-\bi_h\bu]\!],[\![\bu_h]\!]\rangle_e.
	\end{eqnarray}
	
	For $e\in\E_h^0$, define
	\begin{eqnarray*}
		P_e\bf=\frac{1}{|e|}\int_e\bf ds.
	\end{eqnarray*}
	
	 It is easy to check that $P_e[\![\bu_h]\!]=0$ and $[\![\bu]\!]=0$, then by Cauchy-Schwarz inequality, we derive
	\begin{align}\label{ahM2}
		|\langle[\![\bu-\bi_h\bu]\!],[\![\bu_h]\!]\rangle_e|=&
		|\langle[\![\bv]\!]-P_e[\![\bv]\!],[\![\bw]\!]\rangle_e|\nonumber
		\\
		=&|\langle[\![\bv]\!]-P_e[\![\bv]\!],[\![\bw]\!]-P_e[\![\bw]\!]\rangle_e|
		\\
		\lesssim&\big\|[\![\bv]\!]-P_e[\![\bv]\!]\big\|_e\big\|[\![\bw]\!]-P_e[\![\bw]\!]\big\|_e\nonumber,
	\end{align}
	where $\bv=\bu-\bi_h\bu$, $\bw=\bu-\bu_h$.
	
	Assume that $T^+,T^-\in\T_h$ such that $T^+\cap T^-=e$. Let $\bv^+=\bv|_{T^+},\bv^-=\bv|_{T^-}$. By the trace inequality and \eqref{errint}, we have
	\begin{align}\label{ahM3}
		\big\|[\![\bv]\!]-P_e[\![\bv]\!]\big\|_e=&
		\big\|(\bv^+-P_e\bv^+)-(\bv^--P_e\bv^-)\big\|_e\nonumber
		\\
		\lesssim&h^{\frac{1}{2}}|\bv|_{1,T^+\cup T^-}
		\\
		\lesssim& h^{\frac{1}{2}+s}|\bu|_{1+s,T^+\cup T^-}.\nonumber
	\end{align}
	
	Analogously, together with \eqref{errH1} we have
	\begin{align}\label{ahM4}
		\big\|[\![\bw]\!]-P_e[\![\bw]\!]\big\|_e
		\lesssim& \gamma(h)^{-1}h^{\frac{1}{2}+s}|\bu|_{1+s,T^+\cup T^-}.
	\end{align}
	
	Substituting \eqref{ahM2},\eqref{ahM3} and \eqref{ahM4} into \eqref{ahM} we get \eqref{lb1}, which completes the proof.
\end{proof}
	
\textbf{Assumption (A8)}\emph{ Let $(\gamma,\bu)$ and $(\gamma_h,\bu_h)$ be the $j-$th eigenpair of \eqref{vartion1} and \eqref{vartionh}, respectively, then there has
	\begin{eqnarray}\label{assumerr}
		||\bu-\bu_h||_h\gtrsim \gamma(h)^{-1}h^s.
\end{eqnarray}}
\begin{lemma}
	Let $(\gamma,\bu)$ and $(\gamma_h,\bu_h)$ be the $j-$th eigenpair of \eqref{vartion1} and \eqref{vartionh}, respectively. Assume that $\bu\in\bh^{1+s}(\Omega)$ with $0<s<1$ and \eqref{assumerr} holds, then when h is small enough, we have
	\begin{eqnarray*}
		\gamma_h\le\gamma.
	\end{eqnarray*}
\end{lemma}
\begin{proof}
	For the second term in \eqref{expand}, it follows from \eqref{errint} and \eqref{errL2} that
	\begin{eqnarray}\label{lb2}
		||\bi_h\bu-\bu_h||\le ||\bi_h\bu-\bu||+||\bu-\bu_h||\lesssim \gamma(h)^{-2}h^{2s},
	\end{eqnarray}
	for the third term, we have
	\begin{eqnarray}\label{lb3}
		\big|||\bi_h\bu||^2-||\bu||^2\big|=\big|(\bi_h\bu-\bu,\bi_h\bu+\bu)\big|\lesssim ||\bi_h\bu-\bu||\lesssim h^{1+s}.
	\end{eqnarray}
	Substitute \eqref{lb1} and \eqref{assumerr}-\eqref{lb3} into \eqref{expand}, we derive, when h is small enough, 
	\begin{eqnarray*}
		\gamma-\gamma_h\ge0,
	\end{eqnarray*}
	The proof is finished.
\end{proof}
	
	\section{Numerical experiments}In this section, we show some numerical results to verify the analysis in previous sections. Since the exact eigenvalues are unknown, we compute the convergence rate by
\begin{eqnarray*}
	Order=\lg\left(\frac{\gamma_h-\gamma_{\frac{h}{2}}}{\gamma_{\frac{h}{2}}-\gamma_{\frac{h}{4}}}\right)/\lg2.
\end{eqnarray*}

\begin{example}
	Consider the linear elastic eigenvalue problem \eqref{eig} on unit square domain $\Omega=(0,1)^2$ with $\Gamma_N=\phi$. We set the Young's modulus $E$=1, the Poisson ratio $\nu$=0.49, 0.4999, 0.499999. We solve it by $WG$ method and report first 4 discrete eigenfrequencies $\omega_h=\sqrt{\gamma_h}$.
\end{example}
\begin{table}[H]
	\begin{center}
		\caption{k=1}\label{t1}
		\begin{tabular}{|c|c|c|c|c|c|}\hline
				$h$&1/16&1/32&1/64&1/128&Order\\ \hline\hline
		\multicolumn{6}{|c|}{$\nu=0.49$}\\ \hline
				$\omega_{1,h}$&4.133787 & 4.173779 & 4.184683 & 4.187561  &1.92\\ \hline
				$\omega_{2,h}$&5.355592 & 5.473026 & 5.505790 & 5.514496  &1.91\\ \hline
				$\omega_{3,h}$&5.362232 & 5.474950 & 5.506313 & 5.514637  &1.91\\ \hline
				$\omega_{4,h}$&6.335914 & 6.485768 & 6.528086 & 6.539371  &1.91\\ \hline
		\multicolumn{6}{|c|}{$\nu=0.4999$}\\ \hline
				$\omega_{1,h}$&4.122662 & 4.162466 & 4.173337 & 4.176208  &1.92\\ \hline
				$\omega_{2,h}$&5.378227 & 5.496428 & 5.529427 & 5.538199  &1.91\\ \hline
				$\omega_{3,h}$&5.385957 & 5.498625 & 5.530018 & 5.538357  &1.91\\ \hline
				$\omega_{4,h}$&6.330349 & 6.479817 & 6.522102 & 6.533389  &1.91\\ \hline
		\multicolumn{6}{|c|}{$\nu=0.499999$}\\ \hline
			$\omega_{1,h}$&4.122552 & 4.162354 & 4.173224 & 4.176095  &1.92\\ \hline
			$\omega_{2,h}$&5.378404 & 5.496609 & 5.529608 & 5.538381  &1.91\\ \hline
			$\omega_{3,h}$&5.386141 & 5.498807 & 5.530200 & 5.538539  &1.91\\ \hline
			$\omega_{4,h}$&6.330290 & 6.479754 & 6.522039 & 6.533326  &1.91\\ \hline
		\end{tabular}
	\end{center}
\end{table}

Table \ref{t1} and \ref{t2} show the approximate eigenvalues solved by WG method with $k$=1, 2 as the Poisson ratio $\nu$ turns to $\frac{1}{2}$, respectively. As the numerical results represented in the tables, which implies that the eigenfunctions have at least $H^3$-regularity, we can see the WG method not only is locking-free and satisfies the lower bound property, but also has the convergence of $2(k-\delta)$ for eigenvalues approximately, which coincides with our arguments.

\begin{table}[H]
	\begin{center}
		\caption{k=2}\label{t2}
		\begin{tabular}{|c|c|c|c|c|c|}\hline
				$h$&1/16&1/32&1/64&1/128&order\\ \hline\hline
			\multicolumn{6}{|c|}{$\nu=0.49$}\\ \hline
				$\omega_{1,h}$&4.188174 & 4.188547 & 4.188575 & 4.188577 & 3.84 \\  \hline
				$\omega_{2,h}$&5.515875 & 5.517456 & 5.517573 & 5.517581 & 3.85 \\  \hline
				$\omega_{3,h}$&5.515933 & 5.517465 & 5.517573 & 5.517581 & 3.86 \\  \hline
				$\omega_{4,h}$&6.540041 & 6.543122 & 6.543346 & 6.543361 & 3.85 \\  \hline
			\multicolumn{6}{|c|}{$\nu=0.4999$}\\ \hline
				$\omega_{1,h}$&4.176809 & 4.177191 & 4.177219 & 4.177221 & 3.84 \\\hline
				$\omega_{2,h}$&5.539523 & 5.541177 & 5.541299 & 5.541308 & 3.86 \\\hline
				$\omega_{3,h}$&5.539588 & 5.541187 & 5.541300 & 5.541308 & 3.88 \\\hline
				$\omega_{4,h}$&6.533997 & 6.537136 & 6.537365 & 6.537381 & 3.85 \\\hline
			\multicolumn{6}{|c|}{$\nu=0.499999$}\\ \hline
			$\omega_{1,h}$&4.176697 & 4.177078 & 4.177107 & 4.177109 & 3.65 \\
			\hline
			$\omega_{2,h}$&5.539704 & 5.541359 & 5.541481 & 5.541489 & 3.82 \\
			\hline
			$\omega_{3,h}$&5.539769 & 5.541369 & 5.541482 & 5.541489 & 3.84 \\
			\hline
			$\omega_{4,h}$&6.533934 & 6.537073 & 6.537302 & 6.537318 & 3.83 \\ \hline
		\end{tabular}
	\end{center}
\end{table}

\begin{example}
	Consider the linear elastic eigenvalue problem \eqref{eig} on L-shaped  domain $\Omega=(0,2)^2/(1,2)^2$ with $\Gamma_N=\phi$. We set the Young's modulus $E$=1, the poisson ratio $\nu$=0.49, 0.4999, 0.499999 and $\gamma(h)=h^{0.05}$. We solve it by $WG$ method and report first four discrete eigenfrequencies $\omega_h=\sqrt{\gamma_h}$.
\end{example}

\begin{table}[H]
	\begin{center}
		\caption{k=1}\label{t3}
		\begin{tabular}{|c|c|c|c|c|c|}\hline
				$h$&1/16&1/32&1/64&1/128&order\\ \hline\hline
				\multicolumn{6}{|c|}{$\nu=0.49$}\\ \hline
				$\omega_{1,h}$&3.216685 & 3.250269 & 3.261744 & 3.265905 & 1.46 \\\hline
				$\omega_{2,h}$&3.449111 & 3.491878 & 3.503957 & 3.507309 & 1.85 \\ \hline
				$\omega_{3,h}$&3.669315 & 3.704210 & 3.713763 & 3.716342 & 1.89 \\ \hline
				$\omega_{4,h}$&3.989626 & 4.028245 & 4.038850 & 4.041671 & 1.91 \\ \hline
			\multicolumn{6}{|c|}{$\nu=0.4999$}\\ \hline
				$\omega_{1,h}$&3.219344 & 3.253587 & 3.265390 & 3.269713 & 1.45 \\ \hline
				$\omega_{2,h}$&3.453213 & 3.496018 & 3.508099 & 3.511447 & 1.85 \\ \hline
				$\omega_{3,h}$&3.688787 & 3.725146 & 3.735136 & 3.737847 & 1.88 \\ \hline
				$\omega_{4,h}$&3.987055 & 4.026136 & 4.036903 & 4.039774 & 1.91 \\ \hline
		\multicolumn{6}{|c|}{$\nu=0.499999$}\\ \hline
			$\omega_{1,h}$&3.219356 & 3.253605 & 3.265411 & 3.269735 & 1.45 \\ \hline
			$\omega_{2,h}$&3.453240 & 3.496046 & 3.508126 & 3.511474 & 1.85 \\ 
			\hline
			$\omega_{3,h}$&3.688839 & 3.725204 & 3.735196 & 3.737907 & 1.88 \\ \hline
			$\omega_{4,h}$&3.987021 & 4.026106 & 4.036875 & 4.039746 & 1.91 \\ \hline
		\end{tabular}
	\end{center}
\end{table}

Table \ref{t3} and \ref{t4} show the approximate eigenvalues solved by WG method with $k$=1 and 2 as the Poisson ratio $\nu$ turns to $\frac{1}{2}$, respectively. Since the domain is not convex, the eigenfunctions of this problem are sigular. According to \cite{Inzunza2023}, the theoretical convergence order of eigenvalues is $2s\ge1.08$. As we can see from the tables, WG method is locking-free and has the convergence of $\min{2s,2k}-2\delta$. Furthermore, it provides the lower bounds for eigenvalues, which means WG method work well in unconvex domain.

\begin{table}[H]
	\begin{center}
		\caption{k=2}\label{t4}
		\begin{tabular}{|c|c|c|c|c|c|}\hline
			$h$&1/8&1/16&1/32&1/64&order\\ \hline\hline
			\multicolumn{6}{|c|}{$\nu=0.49$}\\ \hline
			$\omega_{1,h}$&3.256093 & 3.263613 & 3.266413 & 3.267662 & 1.43 \\
			 \hline
			$\omega_{2,h}$&3.502326 & 3.507443 & 3.508310 & 3.508517 & 2.56 \\
			 \hline
			$\omega_{3,h}$&3.712912 & 3.716691 & 3.717155 & 3.717264 & 3.03 \\
			 \hline
			$\omega_{4,h}$&4.037738 & 4.042203 & 4.042619 & 4.042668 & 3.43 \\
			 \hline
		\multicolumn{6}{|c|}{$\nu=0.4999$}\\ \hline
			$\omega_{1,h}$&3.259411 & 3.267276 & 3.270230 & 3.271556 & 1.41 \\
			 \hline
			$\omega_{2,h}$&3.506429 & 3.511590 & 3.512451 & 3.512651 & 2.58 \\
			\hline
			$\omega_{3,h}$&3.733890 & 3.738159 & 3.738692 & 3.738820 & 3.00 \\ \hline
			$\omega_{4,h}$&4.035628 & 4.040289 & 4.040734 & 4.040789 & 3.39 \\
			\hline
		\multicolumn{6}{|c|}{$\nu=0.499999$}\\ \hline
			$\omega_{1,h}$&3.259428 & 3.267296 & 3.270252 & 3.271578 & 1.41 \\ \hline
			$\omega_{2,h}$&3.506456 & 3.511617 & 3.512478 & 3.512678 & 2.58 \\ 
			\hline
			$\omega_{3,h}$&3.733948 & 3.738219 & 3.738752 & 3.738881 & 3.00 \\
			\hline
			$\omega_{4,h}$&4.035598 & 4.040261 & 4.040706 & 4.040762 & 3.39 \\
		    \hline
		\end{tabular}
	\end{center}
\end{table}

\begin{example}
	Consider the linear elastic eigenvalue problem \eqref{eig} on unit square domain $\Omega=(0,1)^2$ with $\Gamma_D=\{(x,0):0\le x\le1\}$. We set the Young's modulus $E$=1, the poisson ratio $\nu$=0.49, 0.4999, 0.499999 and $\gamma(h)=h^{0.05}$. We solve it by $WG$ and $CR$ method and report first four discrete eigenfrequencies $\omega_h=\sqrt{\gamma_h}$.
\end{example}
\begin{table}[H]
	\begin{center}
		\caption{WG method, k=1}\label{t5}
		\begin{tabular}{|c|c|c|c|c|c|}\hline
				$h$&1/32&1/64&1/128&1/256&order\\ \hline\hline
		\multicolumn{6}{|c|}{$\nu=0.49$}\\ \hline		
				$\omega_{1,h}$&0.693362 & 0.696795 & 0.698316 & 0.698991 & 1.17 \\ \hline
				$\omega_{2,h}$&1.826880 & 1.832813 & 1.835304 & 1.836371 & 1.22 \\ \hline
				$\omega_{3,h}$&1.856009 & 1.859558 & 1.860488 & 1.860730 & 1.95 \\ \hline
				$\omega_{4,h}$&2.909418 & 2.921006 & 2.925020 & 2.926513 & 1.43 \\ \hline
		\multicolumn{6}{|c|}{$\nu=0.4999$}\\ \hline
				$\omega_{1,h}$&0.695084 & 0.698669 & 0.700272 & 0.700988 & 1.16 \\ \hline
				$\omega_{2,h}$&1.837516 & 1.843765 & 1.846407 & 1.847548 & 1.21 \\ \hline
				$\omega_{3,h}$&1.860667 & 1.864282 & 1.865232 & 1.865478 & 1.94 \\ \hline
				$\omega_{4,h}$&2.904148 & 2.915832 & 2.919931 & 2.921476 & 1.41 \\ \hline
		\multicolumn{6}{|c|}{$\nu=0.499999$}\\ \hline
			$\omega_{1,h}$&0.695101 & 0.698689 & 0.700293 & 0.701006 & 1.16 \\ \hline
			$\omega_{2,h}$&1.837623 & 1.843875 & 1.846519 & 1.847660 & 1.21 \\ \hline
			$\omega_{3,h}$&1.860716 & 1.864331 & 1.865281 & 1.865527 & 1.94 \\ \hline
			$\omega_{4,h}$&2.904096 & 2.915782 & 2.919882 & 2.921427 & 1.41 \\ \hline
		\end{tabular}
	\end{center}
\end{table}
\begin{table}[H]
	\begin{center}
		\caption{CR method}\label{t6}
		\begin{tabular}{|c|c|c|c|c|c|}\hline
			$h$&1/32&1/64&1/128&1/256&order\\ \hline\hline
			\multicolumn{6}{|c|}{$\nu=0.49$}\\ \hline
			$\omega_{1,h}$&0.695688 & 0.697821 & 0.698771 & 0.699193 & 1.17 \\ \hline
			$\omega_{2,h}$&1.831372 & 1.834622 & 1.836057 & 1.836693 & 1.17 \\ \hline
			$\omega_{3,h}$&1.860071 & 1.860625 & 1.860764 & 1.860801 & 1.94 \\ \hline
			$\omega_{4,h}$&2.921482 & 2.924809 & 2.926307 & 2.926982 & 1.15 \\ \hline
		\multicolumn{6}{|c|}{$\nu=0.4999$}\\ \hline
			$\omega_{1,h}$&0.697483 & 0.699738 & 0.700750 & 0.701202 & 1.16 \\ \hline
			$\omega_{2,h}$&1.842195 & 1.845662 & 1.847203 & 1.847892 & 1.16 \\ \hline
			$\omega_{3,h}$&1.864782 & 1.865366 & 1.865513 & 1.865551 & 1.95 \\ \hline
			$\omega_{4,h}$&2.916171 & 2.919657 & 2.921241 & 2.921961 & 1.14 \\ \hline
		\multicolumn{6}{|c|}{$\nu=0.499999$}\\ \hline
			$\omega_{1,h}$&0.697501 & 0.699758 & 0.700771 & 0.701230 & 1.15 \\ \hline
			$\omega_{2,h}$&1.842304 & 1.845773 & 1.847314 & 1.848006 & 1.16 \\ \hline
			$\omega_{3,h}$&1.864831 & 1.865415 & 1.865561 & 1.865602 & 1.89 \\ \hline
			$\omega_{4,h}$&2.916119 & 2.919607 & 2.921191 & 2.921913 & 1.14 \\ \hline
		\end{tabular}
	\end{center}
\end{table}

Table \ref{t5} present the results provided by WG method with $k$=1 in convex domain with mixed boundary, while talbe \ref{t6} show the results by CR method. According to \cite{Meddahi2013}, some eigenfunctions are sigular. The theoretical convergence order of eigenvalues is $2s\ge1.20$ when $\nu=0.49$, for example. It can be seen from the tables, WG method and CR method are both  locking-free and able to provide the lower bound for eigenvalues.
	
	\section{Appendix}In this section, we give a standard error analysis for the weak Galerkin scheme \eqref{sourceh}.

Consider the linear elasticity equation:
\begin{align}
	\left\{\begin{array}{rcl}
		-\nabla\cdot\sigma(\bu)&=&\bf,~~\text{in}\quad\Omega,\\
		\bu&=&\b0,~~\text{on}\quad\Gamma_D\subset\partial\Omega,\\
		\sigma(\bu)\bn&=&\b0,~~\text{on}\quad\Gamma_N\subset\partial\Omega,
	\end{array}\right.
\end{align}
where $\bf\in\bl^2(\Omega)$ is a given function. It has been proved in \cite{Bren1992} that when $\Gamma_N=\phi$ and $\Omega$ is convex, the solution $\bu$ satisfies
\begin{eqnarray*}
	\|\bu\|_2+\lambda\|\di\bu\|_1\lesssim \|\bf\|,
\end{eqnarray*}
where the hidden constant $C$ is independent of $\lambda$.

Furthermore, the following elliptic regularity estimate is used in \cite{Han2002}.
\begin{eqnarray}\label{reg}
	\|\bu\|_{k+1}+\lambda\|\di\bu\|_k\lesssim \|\bf\|_{k-1}.
\end{eqnarray}

Suppose $\bu$ is the solution of \eqref{source},and $\bu_h$ is the numerical solution of \eqref{sourceh}. Denote by $\be_h$ the error that
\begin{center}
	$\be_h=Q_h\bu-\bu_h=\{Q_0\bu-\bu_0,Q_b\bu-\bu_b\}$.
\end{center}
Then we have the following error equation.
\begin{lemma}
	For the error $\be_h$ defined above, we have
	\begin{eqnarray}\label{errequ}
		a_w(\be_h,\bv_h)=\varphi(\bu,\bv_h)+\xi(\bu,\bv_h)+s(Q_h\bu,\bv_h),\quad\forall \bv_h\in V_h,
	\end{eqnarray}
	where
	\begin{align*}
		\varphi(\bu,\bv_h)&=2\mu\sum\limits_{T\in\mathcal{T}_h}\la \bv_0-\bv_b,(\varepsilon(\bu)-\dQ_h(\varepsilon(\bu)))\bn\ra,
		\\
		\xi(\bu,\bv_h)&=\lambda\sum\limits_{T\in\mathcal{T}_h}\la (\bv_0-\bv_b)\cdot\bn,\di\bu-\Q_h(\di\bu)\ra.
	\end{align*}
\end{lemma}
\begin{proof}
	From the definition \eqref{defw1}-\eqref{defw2} and Lemma \ref{exchange}, there holds on each element $T\in\mathcal{T}_h$ that
	\begin{align}\label{1}
			(\varepsilon_w(Q_h\bu),\varepsilon_w(\bv_h))_T=&
			(\dQ_h(\varepsilon(u)),\varepsilon_w(\bv_h))_T\nonumber
			\\
			=&(\dQ_h(\varepsilon(\bu)),\varepsilon(\bv_0))_T+\la \bv_b-\bv_0,\dQ_h(\varepsilon(\bu))\bn\ra
			\\
			=&(\varepsilon(\bu),\varepsilon(\bv_0))_T+\la \bv_b-\bv_0,\dQ_h(\varepsilon(\bu))\bn\ra.\nonumber
	\end{align}	
	Similarly, we have
	\begin{align}\label{2}
			(\nabla_w\cdot Q_h\bu,\nabla_w\cdot\bv_h)_T=&
			(\Q_h(\di\bu),\nabla_w\cdot\bv_h)_T\nonumber
			\\
			=&(\Q_h(\di\bu),\di\bv_0)_T+\la (\bv_b-\bv_0)\cdot\bn,\Q_h(\di\bu)\ra
			\\
			=&(\di\bu,\di\bv_0)_T+\la (\bv_b-\bv_0)\cdot\bn,\Q_h(\di\bu)\ra.\nonumber
	\end{align}
	Summing over all elements and it follows from \eqref{source} that
	\begin{align*}
		a_w(Q_h\bu,\bv_h)=&
		2\mu(\varepsilon_w (Q_h\bu),\varepsilon_w(\bv_h))+\lambda(\nabla_w\cdot Q_h\bu,\nabla_w\cdot\bv_h)+s(Q_h\bu,\bv_h)
		\\
		=&
		2\mu(\varepsilon(\bu),\varepsilon(\bv_0))+\lambda(\di\bu,\di\bv_0)
		+2\mu\la \bv_b-\bv_0,\dQ_h(\varepsilon(\bu))\bn\ra
		\\
		\quad&+\lambda\la (\bv_b-\bv_0)\cdot\bn,\Q_h(\di\bu)\ra+s(Q_h\bu,\bv_h)
		\\
		=&(\bf,\bv_0)+\varphi(\bu,\bv_h)+\xi(\bu,\bv_h)+s(Q_h\bu,\bv_h),
	\end{align*}
	which combines with $\eqref{sourceh}$ completes the proof.
\end{proof}

Then we give the estimates for $\varphi(\bu,\bv_h)$, $\xi(\bu,\bv_h)$ and $s(Q_h\bu,\bv_h)$ as follows.
\begin{lemma}$\label{errest}$
	Suppose that $\bu\in \bh^{k+1}(\Omega)$ and \eqref{reg} holds, then for any $\bv_h\in V_h$, we have
	\begin{align*}
		|\varphi(\bu,\bv_h)|\lesssim& h^k\|\bf\|_{k-1}\|\bv_h\|_V,\\
		|\xi(\bu,\bv_h)|\lesssim&h^k\|\bf\|_{k-1}\|\bv_h\|_V,\\
		|s(Q_h\bu,\bv_h)|\lesssim&\gamma(h)h^k\|\bf\|_{k-1}\|\bv_h\|_V.
	\end{align*}
\end{lemma}
\begin{proof}
	By the Cauchy-Schwarz inequality, the trace inequality and Lemma \ref{projection}, for the first inequality we arrive at
	\begin{align*}
		|\varphi(\bu,\bv_h)|=&
		2\mu\left|\sumT\la \bv_b-\bv_0,(\dQ_h(\varepsilon(\bu))-\varepsilon(\bu))\bn\ra\right|
		\\
		\lesssim& 
		\left(\sumT h_T||\mathbb{Q}_h(\varepsilon(\bu))-\varepsilon(\bu)||_{\partial T}^2\right)^{\frac12}\left(\sumT h_T^{-1}||\bv_0-\bv_b||_{\partial T}^2\right)^{\frac12}
		\\
		\lesssim& 
		h^k\|\bu\|_{k+1}\|\bv_h\|_V.
	\end{align*}
	Similarly, for the second inequality, we have
	\begin{align*}
		|\xi(\bu,\bv_h)|=&
		\lambda\left|\sumT\la (\bv_b-\bv_0)\cdot\bn,\Q_h(\di\bu)-\di\bu\ra\right|
		\\
		\lesssim&
		\lambda\left(\sumT h_T||\Q_h(\di\bu)-\di\bu||_{\partial T}^2\right)^{\frac12}\left(\sumT h_T^{-1}||\bv_0-\bv_b||_{\partial T}^2\right)^{\frac12}
		\\
		\lesssim& 
		h^k(\lambda\|\di\bu\|_k)\|\bv_h\|_V.
	\end{align*}
	As to the last inequality, we again use the Cauchy-Schwarz inequality, the trace inequality and Lemma \ref{projection} to verify
	\begin{align*}
	|s(Q_h\bu,\bv_h)|=&
	\gamma(h)\left|\sumT h_T^{-1}\la Q_b(Q_b\bu-Q_0\bu),Q_b (\bv_0-\bv_b)\ra\right|
	\\
	=&
	\gamma(h)\left|\sumT h_T^{-1}\la Q_b(\bu-Q_0\bu),Q_b (\bv_0-\bv_b)\ra\right|
	\\
	\lesssim&
	\gamma(h)\left(\sumT h_T||\bu-Q_0\bu||_{\partial T}^2\right)^{\frac12}\left(\sumT h_T^{-1}||\bv_0-\bv_b||_{\partial T}^2\right)^{\frac12}
	\\
	\lesssim&
	\gamma(h)h^k\|\bu\|_{k+1}\|\bv_h\|_V,
\end{align*}
    Substitute \eqref{reg} into the estimates above completes the proof.
\end{proof}

With the error equation $\eqref{errequ}$ and the estimates derived in Lemma $\ref{errest}$, we get the following error estimate for the weak Galerkin method.
\begin{lemma}\label{Vnorm}
	Assume the exact solution $\bu\in\bh^{k+1}(\Omega)$, then the following estimate hold true,
	\begin{align}
		||\bu-\bu_h||_V\lesssim\gamma(h)^{-1}h^k\|\bf\|_{k-1}.
	\end{align}
\end{lemma}
\begin{proof}
	Taking $\bv_h=\be_h$ in \eqref{errequ} and it follows from Lemma \ref{errest} that
	\begin{align*}
		\gamma(h)||\be_h||_V^2\lesssim& 
		a_w(\be_h,\be_h)
		\\
		=&
		\varphi(\bu,\be_h)+\xi(\bu,\be_h)+s(Q_h\bu,\be_h)
		\\
		\lesssim&
		h^k\|\bf\|_{k-1}||\be_h||_V,
	\end{align*}
	which implies
	\begin{eqnarray*}
		||\be_h||_V\lesssim\gamma(h)^{-1}h^k\|\bf\|_{k-1}.
	\end{eqnarray*}
	From Lemma $\ref{delta}$ we have
	\begin{eqnarray*}
		||\bu-\bu_h||_V\leq||Q_h\bu-\bu_h||_V+||Q_h\bu-\bu||_V\lesssim \gamma(h)^{-1}h^k\|\bf\|_{k-1}.
	\end{eqnarray*}
	Thus, the proof is completed.
\end{proof}

By using the Nistche's technique, there holds the following $L^2$ error estimate.
\begin{lemma}\label{Xnorm}
	Suppose that $\bu\in\bh^{k+1}(\Omega)$ and the dual problem has $H^2$-regularity, then we have
	\begin{eqnarray*}
		||\bu-\bu_h||_X\lesssim\gamma(h)^{-1}h^{k+1}\|\bf\|_{k-1}.
	\end{eqnarray*}
\end{lemma}

	\bibliographystyle{siam}
\bibliography{library}
\end{document}